\documentclass[11pt]{amsart}
\usepackage{amsfonts,amssymb,amsthm,amsmath,latexsym,graphicx, color}
\usepackage[all]{xy}

\newtheorem{thm}{Theorem}[section]

\newtheorem{cor}[thm]{Corollary}
\newtheorem{lem}[thm]{Lemma}
\newtheorem{prop}[thm]{Proposition}
\newtheorem{exa}[thm]{Example}
\newtheorem{defi}[thm]{Definition}

\newtheorem{rek}[thm]{Remark}

\theoremstyle{definition}

\newcommand{\HP}{\mathrm{HP}}
\newcommand{\NP}{\mathrm{NP}}
\newcommand{\GNP}{\mathrm{GNP}}
\newcommand{\Tr}{\mathrm{Tr}}
\newcommand{\ord}{\mathrm{ord}}

\renewcommand{\bar}{\overline}

\newcommand{\A}{\mathbb{A}}
\newcommand{\C}{\mathbb{C}}
\newcommand{\F}{\mathbb{F}}
\newcommand{\R}{\mathbb{R}}
\newcommand{\Q}{\mathbb{Q}}
\newcommand{\Z}{\mathbb{Z}}

\newcommand{\NN}{\mathbb{N}}
\newcommand{\cM}{\mathcal{M}}

\begin{document}
\title[Newton polygons for a variant of the Kloosterman family]
{Newton polygons for a variant of the Kloosterman family}

\author[Bellovin]{Rebecca Bellovin}
\address{
Department of Mathematics,
450 Serra Mall, Bldg. 380,
Stanford, CA 94305
}\email{rmbellov@math.stanford.edu}

\author[Garthwaite]{Sharon Anne Garthwaite}
\address{Department of Mathematics, Bucknell University, Lewisburg, PA 17837}\email{sharon.garthwaite@bucknell.edu}

\author[Ozman]{Ekin Ozman}
\address{
University of Texas-Austin,
Department of Mathematics,
Austin-TX, 78712
}
\email{ozman@math.utexas.edu}

\author[Pries]{Rachel Pries}
\address{
Department of Mathematics,
Colorado State University,
Fort Collins, CO 80523-1874, USA
}
\email{pries@math.colostate.edu}

\author[Williams]{Cassandra Williams}
\address{
Department of Mathematics and Statistics, James Madison University, Harrisonburg, VA 22802, USA
}
\email{willi5cl@jmu.edu}

\author[Zhu]{Hui June Zhu}
\address{Department of Mathematics, State University of New York, Buffalo, NY 14260}
\email{hjzhu@math.buffalo.edu}

\date{\today}

\thanks{This project was initiated at the workshop WIN Women in Numbers in November 2011.  The authors would like to thank the Banff International Research Station for hosting the workshop and the Fields Institute, the Pacific Institute for the Mathematical Sciences, Microsoft Research, and University of Calgary for their financial support.  Author Bellovin was supported by an NSF Graduate Research Fellowship.  Author Pries was partially supported by
NSF grant DMS-11-01712. Author Zhu was partially supported by NSA grant 1094132-1-57192. 
 We would like to thank the referee for suggestions which improved Section \ref{Snp}. }

\newpage

\maketitle

\begin{abstract}
We study the $p$-adic valuations of roots of $L$-functions associated with certain families of 
exponential sums of Laurent polynomials $f \in \F_q[x_1^{\pm 1}, \ldots, x_n^{\pm 1}]$.
The families we consider are reflection and Kloosterman variants of diagonal polynomials.
Using decomposition theorems of Wan, we determine the Newton and Hodge polygons of a
non-degenerate Laurent polynomial in one of these families.
\end{abstract}

\section{Introduction}
Let $q$ be a power of a prime $p$ and $\F_q$ be the finite field of $q$ elements.  Let $\zeta_p \in \C$ be a fixed primitive $p$th root of unity.
For $k \in \NN$, consider the trace homomorphism ${\rm Tr}_k: \F_{q^k} \to \F_p$.
Given a Laurent polynomial $f(x_1,\ldots, x_n) \in \F_q[x_1^{\pm 1},\ldots,x_n^{\pm 1}]$, 
its $k$-th exponential sum is
\[S_k^*(f)=\sum\limits_{x_i \in {\mathbb F}^*_{q^k}} \zeta_p^{\Tr_kf(x_1,\ldots, x_n)}\in \Q(\zeta_p).\]
The $L$-function of the exponential sum of $f$ is defined as 
\[L^*(f,T)=\exp \left( \sum\limits_{k=1}^\infty S^*_k(f) \frac{T^k}{k}\right).\]
A theorem of Dwork-Bombieri-Grothendieck states that 
\[L^*(f,T)=\frac{\prod_{i=1}^{d_1}(1-\alpha_iT)}{\prod_{j=1}^{d_2}(1-\beta_jT)},\]
where $\alpha_i, \beta_j$ are non-zero algebraic integers for $1 \leq i \leq d_1$ and $1 \leq j \leq d_2$.
Thus \[S_k^*(f)=\beta_1^k + \cdots + \beta_{d_2}^k - \alpha_1^k - \cdots - \alpha_{d_1}^k.\]
The values $d_1$ and $d_2$ depend on geometric and cohomological properties of the motive defined by $f$.
A theorem of Deligne \cite{Deligne} implies that the complex absolute values satisfy 
$|\alpha_i|=q^{u_i/2}$ and $|\beta_j|=q^{v_j/2}$ for some weights $u_i, v_j \in \Z \cap [0,2n]$.  
Also, for each prime $\ell \neq p$, the values $\alpha_i, \beta_j$ are $\ell$-adic units.

There are many open questions about the $p$-adic valuation of the roots and poles of $L^*(f, T)$.
Write $|\alpha_i|_p=q^{-r_i}, |\beta_j|_p=q^{-s_j}$, where the $p$-adic valuation is normalized such that $|q|_p=1/q$.  
Deligne's integrality theorem implies that $r_i,s_j \in \Q \cap [0,n]$.
If $f$ is diagonal, then $\alpha_i, \beta_j$ are roots of products of Gauss sums and the slopes $r_i, s_j$ can be determined using Stickelberger's theorem.
In this paper, we use Wan's decomposition theory \cite{Wan} to study two families of Laurent polynomials that are not diagonal.
We briefly explain the results, referring to Section \ref{Snot} for definitions and background material.

Given a Laurent polynomial $f$, one can define its Newton polytope $\Delta$ 
which is an $n$-dimensional integral convex polyhedron in $\R^n$ determined by the dominant terms of $f$.  
Using $\Delta$, one can define a non-degeneracy condition on $f$.
Also, one can assign a weight function to lattice points of $\R^n$. 
One can associate to $\Delta$ its Hodge numbers and Hodge polygon $\HP(\Delta)$,
a lower convex polygon in $\R^2$ starting at the origin, by counting the number of lattice points of a given weight.

If $f$ is non-degenerate and $\Delta$ is general enough, 
then $L^*(f, T)^{(-1)^{n-1}}$ is a polynomial of degree $n!V(\Delta)$ by results of Adolphson and Sperber \cite{AS}.  
In this case, information about the $p$-adic valuations of the roots of $L^*(f, T)^{(-1)^{n-1}}$ is encapsulated in the Newton polygon $\NP(f)$, 
another lower convex polygon in $\R^2$ starting at the origin.
Grothendieck's specialization theorem implies that there exists a generic Newton polygon 
$\GNP(\Delta,\bar\F_p):=\inf_{f} \NP(f)$ where $f$ ranges over all non-degenerate
Laurent polynomials over $\bar\F_p$ with Newton polytope $\Delta$. 
If $f$ is nondegenerate and ${\rm dim}(\Delta)=n$, then by \cite{AS}, the endpoints of the three polygons meet and
\[\NP(f) \geq \GNP(\Delta, \bar\F_p)\geq \HP(\Delta).\]

There are important theorems and open questions about when $\NP(f) = \HP(\Delta)$ or $\GNP(\Delta, \bar\F_p) = \HP(\Delta)$, 
e.g., \cite{AS}, \cite{Wan93}.
In this paper, we consider two families of Laurent polynomials $f$ that are deformations of diagonal polynomials.
In Section \ref{Snp}, we apply Wan's decomposition theory \cite{Wan} to determine congruence conditions on $p$ for which 
$\NP(f) = \HP(\Delta)$. 
In Section \ref{HodgeNumb}, we compute the Hodge numbers of $\HP(f)$ under certain numeric restrictions.     

Here are the two families we consider.  Fix $\vec{m}=(m_1, \ldots, m_n) \in \NN^n$ and let $f_{n, \vec{m}} = x_1^{m_1}+ \cdots + x_n^{m_n}$.
For $1 \leq j \leq n$, define
\[G_{n, \vec{m}}^j=f_{n,\vec{m}} +x_1^{-m_1}+ \cdots + x_j^{-m_j},\]
and 
\[K_{n,\vec{m}}^j=f_{n,\vec{m}} + (x_1 \cdots x_j)^{-1}.\]

An effective lower bound for the Newton polygon for $\NP(G_{n,m}^j)$ is given by Hodge-Stickelberger polygon
as described in \cite[Theorem 6.4]{BFZ08}, see also further results in \cite{Blache}.
We say that $f \in \F_q[x_1^{\pm 1},\ldots,x_n^{\pm 1}]$ is a {\it reflection variant} of $f_{n,\vec{m}}$ if $\Delta(f)=\Delta(G_{n, \vec{m}}^j)$ for some $1 \leq j \leq n$.
We say that $f \in \F_q[x_1^{\pm 1},\ldots,x_n^{\pm 1}]$ is a {\it Kloosterman variant} of $f_{n,\vec{m}}$ if $\Delta(f)=\Delta(K_{n, \vec{m}}^j)$ for some $1 \leq j \leq n$.

Our motivation to study this problem came from the information that it yields about Newton polygons of varieties defined over $\F_q$.
Consider the affine toric Artin-Schreier variety $V_f$ in $\mathbb{A}^{n+1}$ defined by the affine equation
$y^p-y=f(x_1,\ldots x_n)$ where $f(x_1,\ldots,x_n) \in \F_q[x_1^{\pm 1}, \ldots, x_n^{\pm 1}]$ as above. 
The $p$-adic Newton polygons of $L(f/{\mathbb F}_q, T)$ and $L(V_f/{\mathbb F}_q, T)$ 
are the same after scaling by a factor of $p-1$, denoted by $\NP(V_f)=(p-1){\NP}(f)$.

Further decomposition methods for Newton polygons are developed in \cite{Le09}.  Other related work can be found in \cite{H2}, \cite{H1}. 

\section{Background material} \label{Snot} 

Consider a Laurent polynomial $f \in \F_q[x_1^{\pm 1}, \ldots, x_n^{\pm 1}]$.  Then $f$ is of the form
$f = \sum_{j=1}^{J}a_j\vec{x}^{V_j}$ where $a_j \neq 0$, $V_j=(v_{1,j},\ldots,v_{n,j}) \in \Z^n$, and $\vec{x}^{V_j}:=x_1^{v_{1,j}} \ldots x_n^{v_{n,j}}$ for $1 \leq j \leq J$.

\subsection{The Hodge polygon}

The {\em Newton polytope} $\Delta(f)$ of $f$ is the convex polygon generated by the origin $\vec{0}$ and the lattice points $V_j$.
Note that $\Delta$ is an integral polytope, namely its vertices have integral coordinates.
Without loss of generality, we assume that ${\rm dim}(\Delta)=n$.  Let $V(\Delta)$ denote the volume of $\Delta$.
If $\delta$ is a subset of $\Delta(f)$, let $f^{\delta}=\sum\limits_{V_j \in \delta} a_jx^{V_j}$. 

\begin{defi}
A Laurent polynomial $f$ is {\it non-degenerate} with respect to $\Delta$ and $p$
if for each closed face $\delta$ of $\Delta(f)$ not containing $\vec{0}$, 
the partial derivatives $\{\frac{\partial f^{\delta}}{\partial x_1}, \cdots,\frac{\partial f^{\delta}}{\partial x_n}\}$ have no common zeros with $x_1\cdots x_n \neq 0$ over $\bar{\F}_q$.
\end{defi}

Let $\A(\Delta)$ denote the space of all Laurent polynomials with Newton polytope $\Delta$,
parametrized by their (non-vertex) coefficients $(a_j)$.  It is a smooth irreducible affine variety defined over $\F_p$.
The subspace $\cM_p(\Delta) \subset \A(\Delta)$ of all nondegenerate Laurent polynomials
is the complement of a discriminant locus in $\A(\Delta)$.
It is known that $\cM_p(\Delta)$ is Zariski dense and open in $\A(\Delta)$ for each prime $p$;
in other words, a generic Laurent polynomial with Newton polytope $\Delta$ is non-degenerate.
We assume throughout that $f \in \cM_p(\Delta)$.

\begin{defi}
\begin{enumerate}
\item The {\it cone} $C(\Delta)=\sum_{v\in\Delta}v\R^{\geq 0}$ of $\Delta$ is the monoid generated by vectors in $\Delta$.
\item If $\delta$ is a codimension one face of $\Delta$, with equation $\sum_{i=1}^n c_i x_i=1$ for $c_i\in\Q$, the {\it denominator} $D(\delta)$ 
is ${\rm min}\{d \mid dc_i \in \Z, \ 1 \leq i \leq n\}$.
\item The {\it denominator} $D(\Delta)$ is the least common multiple of $D(\delta)$ for all codimension one faces $\delta$ of $\Delta$ not containing $\vec{0}$.
\item If $u=(u_1, \ldots, u_n) \in \Q^n$, the {\it weight} $w(u)$ is the smallest $c \in \Q^{\geq 0}$ such that $u \in c\Delta:=\{c\vec{x} \mid \vec{x} \in \Delta\}$.
(If there is no such rational number $c$, then $w(u)=\infty$).
\end{enumerate}
\end{defi}

The weight $w(u)$ is finite if and only if $u \in C(\Delta)$.  Here is an equivalent way to define the weight.  
If $u \in C(\Delta)$, then the ray $u \R^{\geq 0}$ intersects a codimension one face of $\Delta$ not containing $\vec{0}$.  
If $\sum_{i=1}^n c_i x_i=1$ is the equation of $\delta$, then $w(u)=\sum_{i=1}^{n} c_iu_i$.
Thus $w(u) \in \frac{1}{D(\delta)}\Z^{\geq 0}$. 

We now define the Hodge numbers by counting the number of lattice points of a given weight $k/D$.

\begin{defi} \label{DHodgenum} If $k \in \Z^{\geq 0}$, 
\begin{enumerate}
\item let $W_\Delta(k)=\#\{u \in \Z^n \mid w(u)=\frac{k}{D(\Delta)}\}$ be the number of lattice points in $\Z^n$ with weight $k/D(\Delta)$. 
\item let $H_\Delta(k)=\sum_{i=0}^{n}(-1)^i\dbinom{n}{i}W_\Delta(k-iD(\Delta))$ (the Hodge number).
\end{enumerate}
\end{defi}

For example, when $n=2$, $H_\Delta(k)=W_\Delta(k)-2W_\Delta(k-m)+W_\Delta(k-2m)$.
The Hodge number $H_\Delta(k)$ is the number of lattice points of weight $k/D(\Delta)$ in a fundamental domain of $\Delta$ 
which corresponds to a basis of the p-adic cohomology used to compute the $L$-polynomial.  
Therefore, $H_\Delta(k) \geq 0$ if $k \geq 0$ and $H_\Delta(k)=0$ if $k>nD(\Delta)$. 
Furthermore, \[\sum_{k=0}^{nD(\Delta)}H_\Delta(k)=n!V(\Delta).\]

\begin{defi}\label{HPvertex}
The Hodge polygon $\HP(\Delta)$ is the lower convex polygon in $\R^2$ that starts at $\vec{0}$ and has a side of slope $k/D$ with horizontal length
$H_{\Delta}(k)$ for $0 \leq k \leq nD$.
In other words, it is the polygon with vertices at the origin and, for $0 \leq j \leq nD$, at the point
\[\left(\sum_{k=0}^{j}H_\Delta(k), \frac{1}{D(\Delta)}\sum_{k=0}^{j}kH_\Delta(k)\right).\]
\end{defi}

\subsection{Newton polygon}
When $f$ is nondegenerate with respect to $\Delta(f)$, then $L^*(f,T)^{(-1)^{n-1}}$ is a polynomial of degree $N=n!V(\Delta)$ \cite[Corollary~3.14]{AS}.
Write $L^*(f,T)^{(-1)^{n-1}}=1+C_1T+\cdots+C_NT^N$ with $C_i \in \Z[\zeta_p]$.
For $C \in \Z[\zeta_p]$, write ${\rm ord}_q(C) = {\rm ord}_p(C)/{\rm log}_p(q)$ where $|C|_p = p^{-{\rm ord}_p(C)}$.
The $p$-adic {\em Newton polygon} $\NP(f)$ of $f$ is the lower convex hull in $\R^2$ of the points $(i,\ord_q(C_i))$ for $0 \leq i \leq N$.
The Newton polygon $\NP(f)$ has a segment with slope $\alpha$ and horizontal length $\ell_\alpha$ 
if and only if $L^*(f,T)^{(-1)^{n-1}}$ has a root of $p$-adic valuation $r_i=\alpha$ with multiplicity $\ell_\alpha$.
Results about the slopes of the Newton polygon of $f$ yield results about the 
$p$-adic Riemann hypothesis on the distribution of the roots of $L^*(f,T)^{(-1)^{n-1}}$ in $\bar\Q_p$.

By Grothendieck's specialization theorem, for each prime $p$, there exists
a generic Newton polygon $\GNP(\Delta,\bar\F_p):=\inf_{f} \NP(f)$ where $f$ ranges over all $f \in \cM_p(\Delta)$ defined over $\bar\F_p$. 

\begin{thm}\label{T:AS} 
\cite[Corollary 3.11]{AS}
If $p$ is prime and if $f \in \cM_p(\Delta)$, then the endpoints of the three polygons meet and
\[\NP(f)\geq \GNP(\Delta;\bar\F_p)\geq \HP(\Delta).\]
\end{thm}

It is natural to ask what the slopes of $\GNP(\Delta,p)$ are and how they vary with $p$.  
In particular, it is natural to ask for which $\Delta$ and $p$ the generic Newton polygon equals the Hodge polygon.
Also, one would like to understand when the Newton polygon of $f$ equals the Hodge polygon.
In this context, Wan proved:

\begin{thm} \cite[Theorem 3]{Wan93}
There is a computable integer $D^*(\Delta) \equiv 0 \bmod D(\Delta)$ such that if $p \equiv 1 \bmod D^*(\Delta)$
then $\GNP(\Delta,\bar\F_p)=\HP(\Delta)$.
\end{thm}

A non-degenerate Laurent polynomial $f$ is {\it ordinary} if $\NP(f)=\HP(\Delta(f))$.  
In \cite[Theorem 1.8]{Wan}, Wan gives conditions under which $\NP(f)=\HP(f)$ for all $f \in \cM_p(\Delta)$, 
in other words, for which all non-degenerate $f$ with $\Delta(f)=\Delta$ are ordinary.

The proofs of these results are quite deep.  
Wan constructs an overconvergent $\sigma$-module $\mathcal{E}(\Delta)$ of rank $n!\mathbf{V}(\Delta)$ on $\cM_p(\Delta)$ 
such that the $L$-function of any non-degenerate $f$ with Newton polytope $\Delta$ can be computed on the fiber 
$\mathcal{E}(\Delta)_f$ of $\mathcal{E}(\Delta)$ at the corresponding point of $\cM_p(\Delta)$, i.e., 
\[L^\ast(f,T)^{(-1)^{n-1}} = \det(I-T\text{Frob}_f|\mathcal{E}(\Delta)_f).\]
The Newton polygon of $L^\ast(f,T)^{(-1)^{n-1}}$ can be computed from the ``linear algebra data'' $\mathcal{E}(\Delta)_f$.  
A general theorem shows that for a family of $F$-crystals~\cite{Katz} or $\sigma$-modules~\cite{Wan00}, 
the Newton polygon goes up under specialization.  This implies that there is a Zariski dense and open subspace 
$U\subset\cM_p(\Delta)$ such that for every $f\in U$, the Newton polygon of $L^\ast(f,T)^{(-1)^{n-1}}$ equals $\GNP(\Delta,p)$.  

\section{Newton polygons of non-diagonal Laurent polynomials} \label{Snp}

In this section, we apply Wan's decomposition theory to study two families of non-diagonal Laurent polynomials.
A Laurent polynomial $f$ is {\it diagonal} if it is the sum of $n$ monomials and $n={\rm dim}(\Delta(f))$. 
We first survey some results about the diagonal case from \cite[Section 2]{Wan}.
Suppose $f = \sum_{j=1}^{n}a_j\vec{x}^{V_j}$ where $a_j \neq 0$, $V_j=(v_{1,j},\ldots,v_{n,j}) \in \Z^n$, and $\vec{x}^{V_j}:=x_1^{v_{1,j}} \ldots x_n^{v_{n,j}}$ for $1 \leq j \leq n$.
Let $\Delta=\Delta(f)$ and suppose ${\rm dim}(\Delta)=n$.  
We will need the following definition.

\begin{defi} 
The polytope $\Delta$ is {\it indecomposable} if the $(n-1)$-dimensional face generated by $V_1, \ldots, V_n$ contains no lattice points other than its vertices.
\end{defi}

Linear algebra techniques are useful for studying the Hodge polygon in the diagonal case.
Let $M$ be the non-singular $n \times n$ matrix $M=(V_1, \ldots, V_n)$.  
The Laurent polynomial $f$ is non-degenerate with respect to $\Delta$ and $p$ if and only if $p \nmid {\rm det}(M)$.
Integral lattice points $\vec{u}$ of the fundamental domain 
\[\Gamma=\R V_1 + \cdots + \R V_n \bmod \Z V_1 + \cdots + \Z V_n\]
are in bijection with the set $S(\Delta)$ of solutions $\vec{r}=(r_1, \ldots, r_n)$ of $M \vec{r}^T \equiv 0 \bmod 1$ with $r_j \in \Q \cap [0, 1)$.
This bijection preserves size in that the weight $w(\vec{u})$ equals the norm $|\vec{r}|=\sum_{j=1}^n r_i$.
Now $S(\Delta)$ is a finite abelian group under addition modulo $1$.  Let $D^*$ be its largest invariant factor.
Consider the multiplication-by-$p$ automorphism $[p]$ on $S(\Delta)$, denoted $\vec{r} \to \{p\vec{r}\}$.
The automorphism $[p]$ is weight-preserving if  $p \equiv 1 \bmod D^*$.

Using Gauss sums and the Stickelberger theorem, one proves that the $p$-adic valuation of a root $\alpha$ of $L^*(f/\F_q, T)^{(-1)^{n-1}}$
can be expressed in terms of the average norm of an element $\vec{r} \in S(\Delta)$ under $[p]$ \cite[Corollary 2.3]{Wan}.
Specifically, the horizontal length of the slope $s$ portion of the Newton polygon equals the number of elements $r \in S(\Delta)$ whose average norm is $s$ \cite[Corollary 2.4]{Wan}.
This yields the following.

\begin{thm} \label{Tdiag}  \cite[Section 2.3]{Wan}
Let $\Delta$ be a simplex containing $\vec{0}$ with ${\rm dim}(\Delta)=n$.  Then
\begin{enumerate}
\item 
$\NP(f)=\HP(\Delta)$ for all $f \in \cM_p(\Delta)$ supported only on the interior and vertices of $\Delta$ if  $p \equiv 1 \bmod D^*$.
\item 
$\GNP(\Delta,\bar\F_p)=\HP(\Delta)$ if $p \equiv 1 \bmod D^*$.
\end{enumerate}
\end{thm}

For the main result, we need to strengthen Theorem \ref{Tdiag} in a certain case.
Suppose $\vec{m}=(m_1, \ldots, m_n) \in \NN^n$ and $f_{n, \vec{m}} = x_1^{m_1}+ \cdots + x_n^{m_n}$.
Suppose $f$ is a Laurent polynomial such that $\Delta(f)=\Delta(f_{n, \vec{m}})$.
Notice that $f$ is non-degenerate with respect to
$\Delta$ and $p$ if and only if $p \nmid D^*= {\rm LCM}(m_1, \ldots m_n)$.

\begin{lem} \label{Lbasic}
Let $\Delta=\Delta(f_{n,\vec{m}})$ with $f_{n, \vec{m}} = x_1^{m_1}+ \cdots + x_n^{m_n}$.
\begin{enumerate}
\item 
Suppose $f \in \cM_p(\Delta)$ is supported only on the interior and vertices of $\Delta$.
Then $\NP(f)=\HP(\Delta)$ if and only if $p \equiv 1 \bmod D^*$.
\item 
If $m_1,\ldots,m_n$ are pairwise relatively prime, then 
$\GNP(\Delta,\bar\F_p)=\HP(\Delta)$ if and only if $p \equiv 1 \bmod D^*$.
\end{enumerate} 
\end{lem}

\begin{proof}
\begin{enumerate}
\item The sufficiency statement follows from Theorem \ref{Tdiag}. 
For the other direction, 
if $f$ is ordinary then each boundary restriction $x_i^{m_i}$ is ordinary by Wan's boundary decomposition theorem \cite[Section 5]{Wan93}. 
Hence $p\equiv 1\bmod m_i$ for $ 1 \leq i \leq n$ which implies $p\equiv 1\bmod D^*$.

\item The polytope $\Delta$ is indecomposable if and only if $m_1, \ldots, m_n$ are pairwise relatively 
prime. Then the statement follows from part (1) and Theorem \ref{Tdiag}.
\end{enumerate}
\end{proof}

The facial decomposition theory of Wan allows one to study the Newton polygon of a non-diagonal Laurent polynomial by dividing $\Delta$ into smaller diagonal polytopes.
 
\begin{thm}\label{decomp}\cite[Theorem 8]{Wan93}
Suppose $f$ is non-degenerate and ${\rm dim}(\Delta(f))=n$.  Let $\delta_1,\ldots,\delta_h$ be the codimension $1$ faces of $\Delta(f)$ which do not contain $\vec{0}$.  
Then $f$ is ordinary if and only if $f^{\delta_i}$ is ordinary for each $i$.
\end{thm}

As illustrations of Wan's facial decomposition theory, we study two deformation families of basic diagonal polynomials.

\begin{defi}
Fix $\vec{m}=(m_1, \ldots, m_n) \in \NN^n$ and let $f_{n, \vec{m}} = x_1^{m_1}+ \cdots + x_n^{m_n}$.
A Laurent polynomial $f \in \F_q[x_1^{\pm 1},\ldots,x_n^{\pm 1}]$ is:
\begin{enumerate}
\item a {\it reflection variant} of $f_{n,\vec{m}}$ if $\Delta(f)=\Delta(G_{n, \vec{m}}^j)$ for some $1 \leq j \leq n$ 
where
\[G_{n, \vec{m}}^j=f_{n,\vec{m}} +x_1^{-m_1}+ \cdots + x_j^{-m_j}.\]
\item a {\it Kloosterman variant} of $f_{n,\vec{m}}$ if $\Delta(f)=\Delta(K_{n, \vec{m}}^j)$ for some $1 \leq j \leq n$
where
\[K_{n,\vec{m}}^j=f_{n,\vec{m}} + (x_1 \cdots x_j)^{-1}.\]
\end{enumerate}
\end{defi}

If $n=2$ and $m_1=m_2=1$, then $K_{2, (1,1)}^2$ is the classical Kloosterman polynomial, and it is well-known in this case that the 
Newton polygon has slopes $0$ and $1$ each with multiplicity one.  Pictures and basic facts about the polytopes for $G_{n, \vec{m}}^j$ and $K_{n, \vec{m}}^j$
can be found in Section \ref{HodgeNumb}.  Here is our main result.  

\begin{cor} \label{CreflectionNewton}
Suppose $f$ is a reflection variant or a Kloosterman variant of $f_{n,\vec{m}}$ for some $1 \leq j \leq n$.
Write $\Delta=\Delta(G_{n,\vec{m}}^j)$ or $\Delta=\Delta(K_{n,\vec{m}}^j)$ as appropriate.
\begin{enumerate}
\item Then $f$ is non-degenerate if and only if $p \nmid D^*={\rm LCM}(m_1, \ldots, m_n)$.
\item $\NP(f)=\HP(\Delta)$ for all $f \in \cM_p(\Delta)$ supported only
on the interior and vertices of $\Delta$ 
if and only if $p \equiv 1 \bmod D^*$.
\item If $m_1,\ldots,m_n$ are pairwise relatively prime, 
then $\GNP(\Delta,\bar\F_p)=\HP(\Delta)$ if and only if $p \equiv 1 \bmod D^*$.
\end{enumerate}
\end{cor}

\begin{proof}
This proof follows essentially from Lemma \ref{Lbasic}.
The proof of each part relies on the decomposition of $\Delta$ into different faces. 
By \cite{Wan93}, one can measure whether $f$ is non-degenerate, whether the generic Newton polygon and the Hodge polygon coincide, 
and whether the Newton polygon and the Hodge polygon coincide by seeing whether these properties are true for  
the restriction $f^\delta$ of $f$ to each face $\delta$ of $\Delta$.
  
For the reflection case, after a change of variables of the form $x_i \mapsto x_i^{\pm 1}$, one can restrict to the face of $f_{n, \vec{m}}=G_{n,\vec{m}}^0$ not containing $\vec{0}$.  
The result then follows from Lemma \ref{Lbasic}.

For the Kloosterman case, there is a unique face not containing $-\vec{1}_j=-\sum_{i=1}^j e_i= (-1, \ldots, -1, 0, \ldots, 0)$.
It is the same face as in Lemma \ref{Lbasic}; in particular, $D^*={\rm LCM}(m_1, \ldots, m_n)$ for this face and its vertices are the only lattice points with integral coordinates on 
this face if and only if $m_1, \ldots, m_n$ are pairwise relatively prime.

There are $j$ other faces of $\Delta$ not containing $\vec{0}$.
We consider the face $\delta$ through $-\vec{1}_j$ and $v_i=m_ie_i$ for $2 \leq i \leq n$.  The argument for the other faces is similar.
By Lemma \ref{LdetK}, $\delta$ is contained in the hyperplane
\[\frac{1}{m_2}x_2 + \cdots \frac{1}{m_n}x_{n} - \frac{m+n-1}{m} x_1=1.\]
The integral lattice points $\vec{u}$ of the fundamental domain 
\[\Gamma=\R(-\vec{1}_j) + \R v_2 + \cdots + \R v_n \bmod \Z(-\vec{1}_j)  + \Z v_2 + \cdots + \Z v_n\]
are the set 
\[\{(0, u_2, \ldots, u_n) \in \Z^n \mid 0 \leq u_i < m_i\}.\]
Thus $\Gamma \simeq \times_{i=2}^n \Z/m_i$
and $D_1^*={\rm LCM}(m_2, \ldots, m_n)$ is the largest invariant factor of $\Gamma$.
The multiplication-by-$p$ map on $\Gamma$ is thus weight-preserving if 
$p \equiv 1 \bmod D_1^*$.
Since $D_1^*$ divides $D^*$, the face $\delta$ places no new constraints on the condition $\GNP(\Delta,\bar\F_p)=\HP(\Delta)$.
Furthermore, if $\delta$ does not contain $\vec{0}$, then there are no lattice points on $\delta$ other than the vertices. 
Thus the face $\delta$ places no new constraints on the condition $\NP(f)=\HP(\Delta)$ for all $f \in \cM_p(\Delta)$.

Conversely, if $f$ is ordinary then its restriction to each face $f^\delta$ is ordinary.
Then $p\equiv 1\bmod D^*$ by Lemma \ref{Lbasic}.
\end{proof}

\begin{rek}
By \cite[Corollary~3.14]{AS}, if $f$ is non-degenerate, then $L^*(f, T)^{(-1)^{n-1}}$ is a polynomial of degree $n! V(\Delta)$.
In the reflection case,
\[V(\Delta(G_{n, \vec{m}}^j))=2^jV(G_{n,\vec{m}}^0)=2^j \prod_{j=1}^n m_j/n!.\]
For the Kloosterman case, write $s_k$ for the $k$th symmetric product in $m_1, \ldots, m_j$.  For example, 
$s_j=\prod_{i=1}^j m_i$.  Then, see Lemma \ref{LdetK},
\[V(\Delta(K_{n, \vec{m}}^j))=\left(s_j + \sum_{i=1}^{j-1} (-1)^i i s_{j-1-i}\right)\prod_{i={j+1}}^n m_i/n!.\]
\end{rek}

\section{Computation of Hodge polygons} \label{HodgeNumb}

In this section, we describe the Hodge polygons for two types of Laurent polynomials: the reflection variants $G_{n,\vec{m}}^j$ in Section \ref{Sref};
and the Kloosterman variants $K_{n,\vec{m}}^j$ in Section \ref{Sklo}.  
Each of these is a generalization of the diagonal case which we review in Section \ref{Sdiag}.
We give explicit formulae for the Hodge numbers under certain numeric restrictions on $\vec{m}$.

Fix $n \in \NN$ and $\vec{m}=(m_1,\ldots, m_n) \in \NN^n$.
Let $v_i=m_i \vec{e}_i$ where $\vec{e}_i$ is the standard basis vector of $\R^n$;
in other words, $v_1=(m_1, 0 \ldots, 0)$, $v_2=(0,m_2, 0, \ldots, 0)$, etc.
Write $\vec{x}=(x_1, \ldots, x_n)$.

\subsection{Diagonal Case} \label{Sdiag}

Recall that a Laurent polynomial $f \in \mathbb{F}_q[x_1^{\pm 1},\ldots,x_n^{\pm 1}]$ is diagonal if it is the sum of $n$ monomials and ${\rm dim}(\Delta(f))=n$.  
If $f$ is diagonal, each reciprocal zero of its $L$-function can be computed using Gauss sums, yielding a theoretical understanding of the Newton Polygon of the diagonal case.
The diagonal case is still interesting, however, since nontrivial combinatorial and arithmetic problems arise in computing the Newton Polygon. 

Let $f=\sum\limits_{j=1}^{n}a_j x^{V_j}$, with $a_j \in \F_q$, be a diagonal, non-degenerate Laurent polynomial. 
Let's recall the definition of Gauss sums. 

\begin{defi} Let $\chi$ be the Teichmuller character of $\F_q^*$. For $0 \leq k \leq q-2$, the Gauss sum $G_k(q)$ over $\F_q$ is defined as: 
\[G_k(q)=-\sum\limits_{a \in \F_q^*}\chi(a)^{-k}\zeta_p^{\Tr(a)}.\]
\end{defi}

Gauss sums satisfy certain interpolation relations which yield formulas for the exponential sums $S_k^*(f)$ \cite[16]{Wan}. 
For example, 
\[S_1^*(f)=\sum\limits_{x_j \in \F_q^*} \zeta_p^{\Tr(f(x))} = \displaystyle(-1)^n\sum\limits_{k_1V_1+\ldots+k_nV_n \equiv 0 \bmod q-1} \prod\limits_{i=1}^n \chi(a_i)^k G_{k_i}(q).\] 
Combining this with the Hasse-Davenport relation, Wan obtains an explicit formula for $L^*(f,T)^{(-1)^{n-1}}$ in \cite[Theorem 2.1]{Wan}. 
By applying Stickelberger's Theorem, it is possible to determine the $p$-adic absolute values of the reciprocal zeros of $L^*(f,T)^{(-1)^{n-1}}$. 
In particular, the Newton Polygon is independent of the coefficients $a_j$ and one can suppose $f=\sum\limits_{j=1}^{n}x^{V_j}$ without loss of generality.


We now restrict to the special case of Laurent polynomials of the form $f_{n,\vec{m}}= \sum_{i=1}^n x_i^{m_i}$. 
The vertices of the polytope $\Delta:=\Delta(f_{n,\vec{m}})$ are $\{v_1, \ldots, v_n, \vec{0}\}$ and the volume is $V(\Delta)=\prod_{j=1}^nm_j/n!$.
The denominator is $D(\Delta)={\rm LCM}(m_1, \ldots, m_n)$.
The numeric restriction in Section \ref{Sdiagequi} is that $m_i=m_j$ for all $1 \leq i,j \leq n$ and in Section \ref{Sdiagn=2} is that $n=2$ and ${\rm gcd}(m_1,m_2)=1$.

\subsubsection{General dimension, equilateral} \label{Sdiagequi}
 
For later use, we review some results about the Hodge numbers of the diagonal polynomials
\[G^0_{n,m}=x_1^m+ \cdots + x_n^m.\]

\begin{lem} \label{DiagWk}
The weight numbers for $G^0_{n,m}$ are:
\[
W(k)=\dbinom{n-1+k}{n-1}.
\]
The Hodge numbers for $G^0_{n,m}$ are:
\[H(k)=\sum\limits_{i=0}^{n}(-1)^i\dbinom{n}{i}\dbinom{n-1+k-im}{n-1}.\]
\end{lem}

\begin{proof}
 The face of $\Delta$ not containing $\vec{0}$ is the hyperplane
\[\frac{1}{m}x_1 + \cdots +\frac{1}{m} x_n=1.\]
Thus $D(\Delta)=m$.
The cone $c(\Delta)$ is $\{(a_1, \dots, a_n)\in \R^n \mid a_i \geq 0\}$.
The weight of a vector is given by the formula:
$w(\vec{x})=\frac{1}{m}x_1 + \cdots +\frac{1}{m} x_n$.
The number $W(k)$ of points in $c(\Delta)$ with weight $k/m$ is the number of solutions to
\[
x_1+x_2+\cdots+x_n = k,
\]
which yields the formula for $W(k)$.  The formula for $H(k)$ follows from Definition \ref{DHodgenum}.
\end{proof}

\begin{rek}
The vertices of $\HP(\Delta(G^0_{n,m}))$ are at $(0,0)$ and $(x_j,y_j)$ where
\[x_j = \sum\limits_{i=0}^{\lfloor{j/m}\rfloor}(-1)^i\dbinom{n}{i}\dbinom{n+j-im}{n},\]
and
\[y_j=\frac{1}{m}\sum\limits_{i=0}^{\lfloor{j/m}\rfloor}(-1)^i\dbinom{n}{i} \left(n\cdot\dbinom{n+j-im}{n+1}+im\cdot\dbinom{n+j-im}{n}\right).\]
\end{rek}

\subsubsection{Dimension two, non-equilateral} \label{Sdiagn=2}

Suppose $\vec{m}=(m_1,\ldots, m_n) \in \NN^n$ with $m_1, \ldots m_n \in \NN$ pairwise relatively prime.
Let $W_{n, \vec{m}}^0(k):=W_{\Delta(f_{n, \vec{m}})}(k)$.
Let $M_j:=\prod\limits_{i=1, i\neq j}^{n} m_i$.
Then 
\[W^0_{n, \vec{m}}(k)= \#\{(x_1, \ldots, x_n) \in \NN^n \mid \sum\limits_{i=1}^n M_ix_i=k\}.\]  
These restricted partition functions can be computed using Dedekind sums \cite{Beck}.

Restricting to the case $n=2$, then 
\[W^0_{2, \vec{m}}(k)=\#\{(x_1,x_2) \in \NN^2 \mid m_2x_1+m_1x_2=k \}.\] 
Consider the generating function: 
\[\frac{1}{1-z^{m_1}}\frac{1}{1-z^{m_2}}=\sum\limits_{x_2=0}^{\infty}z^{m_1x_2}\sum\limits_{x_1=0}^{\infty}z^{m_2x_1}=\sum\limits_{k \geq 0}W_{2, \vec{m}}^0(k)z^k.\]
In this case, Popoviciu used partial fractions to give the following formula for $W^0_{2, \vec{m}}(k)$.
For $x \in {\mathbb Q}$, let $\{x\}=\lfloor{x}\rfloor-x$ denote the fractional part of $x$.

\begin{thm} \label{explicitformula}  \cite[Section 1.4]{Beck}
Given $m_1, m_2 \in \NN$ with ${\rm gcd}(m_1,m_2)=1$, let $m_1^{-1}, m_2^{-1} \in \NN$ be such that:
\begin{enumerate}
\item $1 \leq m_1^{-1} < m_2$ and $m_1m_1^{-1} \equiv 1 \bmod m_2$ and\\ 
\item $1 \leq m_2^{-1} < m_1$ and $m_2m_2^{-1} \equiv 1 \bmod m_1$. 
\end{enumerate}
Then \[W_{2, (m_1,m_2)}^0(k)=\frac{k}{m_1m_2}-\left \{ \frac{m_2^{-1}k}{m_1} \right \} - \left \{ \frac{m_1^{-1}k}{m_2} \right \} + 1.\] 
\end{thm}

Using Theorem \ref{explicitformula}, one can explicitly compute all Hodge numbers $H^0(k)$ 
for $W^0_{2,(m_1,m_2)}=x_1^{m_1}+x_2^{m_2}$ when ${\rm gcd}(m_1,m_2)=1$.
Note that the sum of the Hodge numbers is
\[\sum\limits_{k=0}^{2m_1m_2} H^0(k)=m_1m_1=2V(\Delta(f_{2,(m_1,m_2)})).\]


\begin{table}[h] 
	\centering
		\begin{tabular}{||c||c|c|c|c||}
\hline
		$k$ & $0, 1, \ldots, m_1m_2-1 $ & $m_1m_2$ & $m_1m_2+t; 0 < t <m_1m_2$ & $2m_1m_2$ \\ \hline
		$H^0(k)$ & $W_{2, (m_1,m_2)}^0(k)$ & $0$ & $1-W_{2, (m_1,m_2)}^0(t)$ & $0$ \\
		\hline
		\end{tabular}
	\caption{Hodge Numbers for $x_1^{m_1}+x_2^{m_2}$ if ${\rm gcd}(m_1,m_2)=1$}
	\label{Tadiag}
\end{table}

\begin{rek}
The method for $n=2$ can be generalized to higher dimensions; complicated formulas for $W^0_{n, \vec{m}}(k)$ can be found in terms of Dedekind sums \cite[Theorem 1.7]{Beck}. 
For instance, when $n=3$ and $m_1,m_2,m_3$ are pairwise relatively prime then
\begin{eqnarray*}
W^0_{3, (m_1,m_2,m_3)}(k) & = & \frac{k^2}{m_1m_2m_3}+\frac{k}{2}\left ( \frac{1}{m_1m_2} + \frac {1}{m_1m_3} +\frac {1}{m_2m_3} \right ) \\
& + & \frac{1}{12}\left ( \frac{3}{m_1} + \frac {3}{m_2} +\frac {3}{m_3} + \frac{m_1}{m_2m_3} + \frac{m_2}{m_1m_3} +\frac{m_3}{m_1m_1} \right ) \\
& + & \varphi_{m_1}(m_2,m_3)(k) + \varphi_{m_2}(m_1,m_3)(k) + \varphi_{m_3}(m_1,m_2)(k),
\end{eqnarray*}
where $\varphi_{a}(b,c)(k):=\frac{1}{c}\sum \limits_{i=1}^{c-1}[(1-\zeta_c^{ib})(1-\zeta_c^{ia})\zeta_c^ik]^{-1}$. 
\end{rek}

\subsection{Reflection variant Laurent polynomials} \label{Sref}

Suppose $\vec{m}=(m_1, \ldots, m_n) \in \NN^n$ and let
\[G_{n, \vec{m}}^0=x_1^{m_1}+\ldots + x_n^{m_n}.\]  
The polytope $\Delta_{n, \vec{m}}^0$ for $G_{n, \vec{m}}^0$ has vertices $\vec{0}$ and $v_i$ for $1 \leq i \leq n$.

We consider reflections of $\Delta_{n, \vec{m}}^0$ across coordinate hyperplanes.  After a permutation of the variables, 
it is no loss of generality to reflect across the hyperplanes $x_i=0$ for $1 \leq i \leq j$.  
Let \[G_{n, \vec{m}}^j=x_1^{m_1}+\ldots+x_n^{m_n}+x_{1}^{-m_1}+\ldots+x_{j}^{-m_j}.\]
Let $\Delta_{n, \vec{m}}^j$ be the polytope of $G_{n, \vec{m}}^j$.
For example, $G_{n, \vec{m}}^1 = x_1^{m_1}+\ldots + x_n^{m_n}+x_1^{-m_1}$ and $\Delta_{n, \vec{m}}^1$ is the polygon in $\R^n$ with vertices $v_i$ for $1 \leq i \leq n$ and $-v_1$.
Then $\Delta_{n, \vec{m}}^j$ has $n+j$ vertices other than $\vec{0}$ and 
\[{\rm Vol}(\Delta_{n,\vec{m}}^j)=2^j\cdot {\rm Vol}(\Delta_{n,\vec{m}}^0) = 2^j\prod_{i=1}^n m_i/n!. \]

Using the inclusion-exclusion principle, there is a recursive formula for the weight numbers of $\Delta_{n, \vec{m}}^j$:
\begin{equation} \label{Erec}
W_{\Delta_{n, \vec{m}}^j}(k)  =  2 W_{\Delta_{n, \vec{m}}^{j-1}}(k)  -  W_{\Delta_{n-1, (m_1,\ldots,\hat{m}_j, \ldots, m_n)}^{j-1}} (k),
\end{equation}
where the notation $\hat{m}_j$ means that the $j$th variable is omitted.
Using this recursive formula, it is possible to obtain the weights for a general reflection case in terms of the weights for the base case $j=0$.

\subsubsection{General dimension, equilateral}

Suppose $\vec{m}=(m, \ldots, m)$ and write
\[G_{n,m}^j=x_1^m+\ldots+x_n^m+x_1^{-m}+\ldots+x_j^{-m}.\]
The polytope $\Delta_{n,m}^j=\Delta(G_{n,m}^j)$ is obtained by reflecting $\Delta^0_{n,m}$ across the hyperplanes $x_i=0$ for $1 \leq i \leq j$, see 
Figure \ref{fig:example2a_corrected}.

\begin{figure}[h]
	\centering
		\includegraphics[scale=.75]{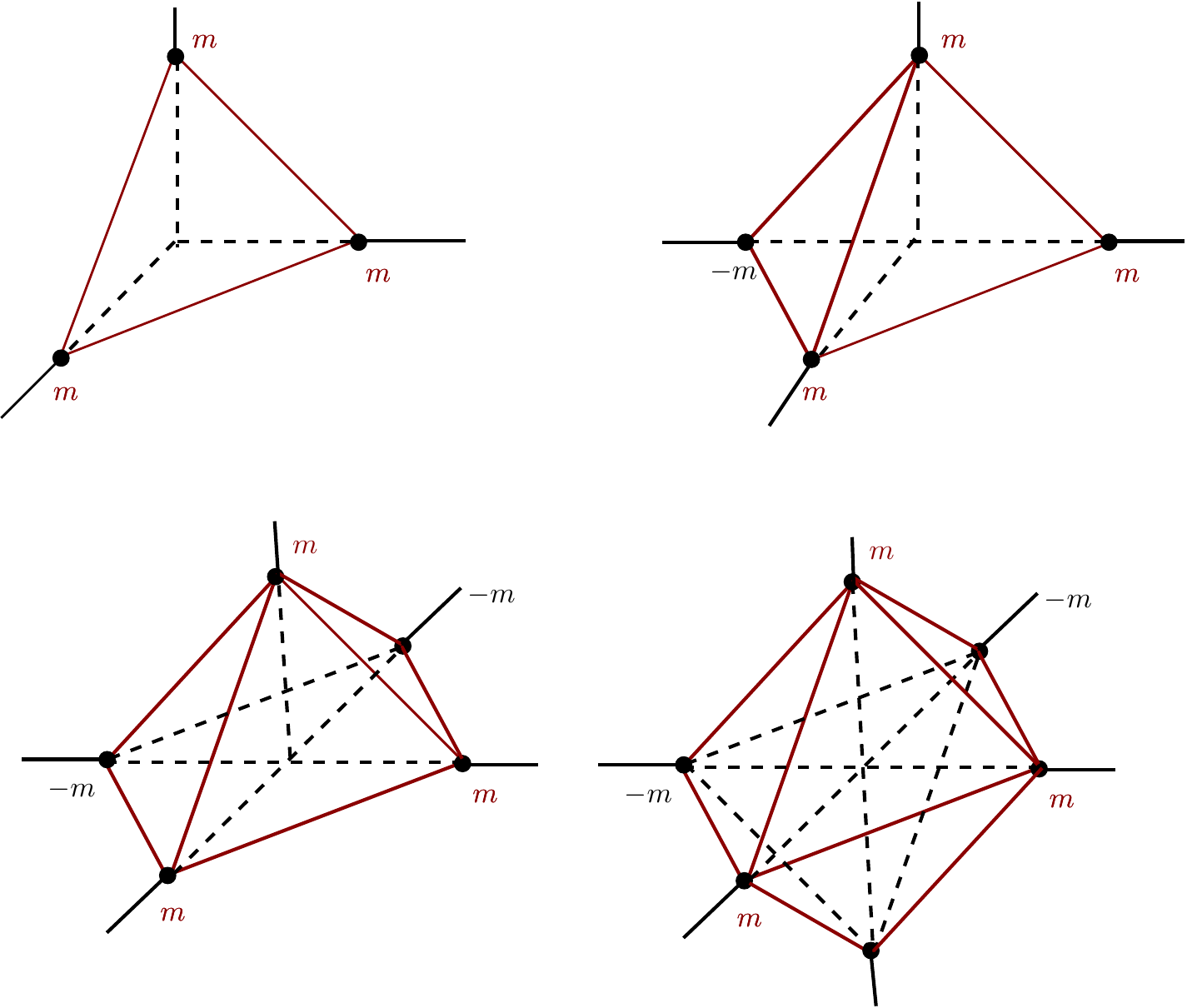}
	\caption{$\Delta_{3,m}^j$ for $0 \leq j \leq 3$}
	\label{fig:example2a_corrected}
\end{figure}

In this case, ${\rm Vol}(\Delta_{n,m}^j)= 2^jm^n/n!$ and Equation (\ref{Erec}) yields the recursive formula
\begin{equation}\label{ReflectionRecursionFormula}
W_{\Delta_{n,m}^j}(k)=2W_{\Delta_{n,m}^{j-1}}(k)-W_{\Delta_{n-1,m}^{j-1}}(k).
\end{equation}

We obtain the following closed form for the weight numbers:

\begin{prop} The weight numbers for $G_{n,m}^j$ are given by:
\[W_{\Delta_{n,m}^j}(k)=\sum\limits_{i=0}^{j} 2^{j-i} (-1)^i  \dbinom{j}{i}W_{\Delta_{n-i,m}^0}(k).\]
\end{prop}

\begin{proof}
First, the formula holds when $j=0$.

To show the formula satisfies the recursion in \eqref{ReflectionRecursionFormula}, we compute
\[
2W_{\Delta_{n,m}^{j-1}}(k)
= \sum\limits_{i=0}^{j-1} 2^{j-i} (-1)^i \dbinom{j-1}{i}W_{\Delta_{n-i,m}^0}(k),\]

and 
\begin{align*}
-W_{\Delta_{n-1,m}^{j-1}}(k)&=\sum_{i=0}^{j-1} 2^{j-1-i} (-1)^{i+1} \dbinom{j-1}{i}W_{\Delta_{n-1-i,m}^0}(k)\\
  &=\sum_{i=0}^{j-1} 2^{j-(i+1)} (-1)^{i+1}  \dbinom{j-1}{i}W_{\Delta_{n-(i+1),m}^0}(k) \\
  &=\sum_{i=1}^{j} 2^{j-i} (-1)^i \dbinom{j-1}{i-1}W_{\Delta_{n-i,m}^0}(k).
\end{align*}

Then $2W_{\Delta_{n,m}^{j-1}}(k)-W_{\Delta_{n-1,m}^{j-1}}(k)$ equals
\begin{align*}
& = \sum\limits_{i=0}^{j-1} 2^{j-i} (-1)^i \dbinom{j-1}{i}W_{\Delta_{n-i,m}^0}(k) + \sum_{i=1}^{j} 2^{j-i} (-1)^i \dbinom{j-1}{i-1}W_{\Delta_{n-i,m}^0}(k)\\
&= 2^j W_{\Delta_{n,m}^0}(k) + \sum_{i=1}^{j-1} 2^{j-i}(-1)^i \left(\dbinom{j-1}{i}+\dbinom{j-1}{i-1}\right) W_{\Delta_{n-i,m}^0}(k) + (-1)^j W_{\Delta_{n-j,m}^0}(k)\\
&=\sum\limits_{i=0}^{j} 2^{j-i} (-1)^i  \dbinom{j}{i}W_{\Delta_{n-i,m}^0}(k)=W_{\Delta_{n,m}^j}(k).
\end{align*}

\end{proof}

\begin{exa} \label{ex.HodgeG}
Weight and Hodge numbers for $G_{2,m}^j$ with $0 \leq j \leq 2$.


\begin{table}[h]
	\centering
		\begin{tabular}{||c||c|c|c|c|c|c|c|c|c||}
\hline
		$k$ & $0$ & $1$ & $\ldots$ & $m-1$ & $m$ & $m+1$ & $\ldots$ & $2m-1$ & $2m$ \\ \hline
		$W(k)$ & $1$ & $2$ & $\ldots$ & $m$  & $m+1$ &  $m+2$ & $\ldots$ & $2m$ & $2m+1$ \\ \hline
		$H(k)$ & $1$ & $2$ & $\ldots$ & $m$  & $m-1$ &  $m-2$ & $\ldots$ & $0$ & $0$\\
		\hline
		\end{tabular}
	\caption{Hodge Numbers for $G_{2,m}^0=x_1^m+x_2^m$}
	\label{HodgeDiag2}
\end{table}

\begin{table}[h]
	\centering
		\begin{tabular}{||c||c|c|c|c|c|c|c|c|c||}
\hline
		$k$ & $0$ & $1$ & $\ldots$ & $m-1$ & $m$ & $m+1$ & $\ldots$ & $2m-1$ & $2m$ \\ \hline
		$W(k)$ & $1$ & $3$ & $\ldots$ & $2m-1$  & $2m+1$ &  $2m+3$ & $\ldots$ & $4m-1$ & $4m+1$ \\ \hline
		$H(k)$ & $1$ & $3$ & $\ldots$ & $2m-1$  & $2m-1$ &  $2m-3$ & $\ldots$ & $1$ & $0$\\
		\hline
		\end{tabular}
	\caption{Hodge Numbers for $G_{2,m}^1=x_1^m+x_2^m+ x_1^{-m}$}
	\label{HodgeG21}
\end{table}

\begin{table}[h]
	\centering
		\begin{tabular}{||c||c|c|c|c|c|c|c|c|c||}
\hline
		$k$ & $0$ & $1$ & $\ldots$ & $m-1$ & $m$ & $m+1$ & $\ldots$ & $2m-1$ & $2m$ \\ \hline
		$W(k)$ & $1$ & $4$ & $\ldots$ & $4(m-1)$  & $4(m+1)$ &  $4(m+3)$ & $\ldots$ & $4(2m-1)$ & $8m$ \\ \hline
		$H(k)$ & $1$ & $4$ & $\ldots$ & $4(m-1)$  & $4m-2$ &  $4(m-1)$ & $\ldots$ & $4$ & $1$\\
		\hline
		\end{tabular}
	\caption{Hodge Numbers for $G_{2,m}^2=x_1^m+x_2^m+x_1^{-m} + x_2^{-m}.$}
	\label{HodgeG22}
\end{table}

\end{exa}

\subsubsection{Dimension two, non-equilateral}



Suppose $n=2$ and ${\rm gcd}(m_1,m_2)=1$.  Let $W^j(k):=W_{\Delta^j_{2, (m_1,m_2)}}(k)$ for $0 \leq j \leq 2$.
Then $W^0(k)$ can be computed using Theorem \ref{explicitformula}.
Equation (\ref{Erec}) gives recursive formulae $W^1(k)=2W^0(k)-1$ and $W^2(k)=2W^1(k)-2=4W^0(k)-4$.

The Hodge numbers are computed in Tables \ref{Tn=2j=1} and \ref{Tn=2j=2}.
Note that the sum of the Hodge numbers in Table \ref{Tn=2j=1} is 
\[\sum\limits_{k=0}^{2m_1m_2}H^1(k)=2m_1{m_2}=2V(\Delta^1_{2,(m_1,{m_2})}),\]
and in Table \ref{Tn=2j=2} is
\[\sum\limits_{k=0}^{2m_1m_2}H^2(k)=4m_1{m_2}=2V(\Delta^2_{2, (m_1,{m_2})}).\]

\begin{table}[h]
	\centering
		\begin{tabular}{||c||c|c|c|c||}
\hline
		$k$ & $0, 1, \ldots, m_1{m_2}-1 $ & $m_1{m_2}$ & $m_1{m_2}+t; 0 < t <m_1{m_2}$ & $2m_1{m_2}$ \\ \hline
		$H^1(k)$ & $2W^0(k)-1$ & $1$ & $3-2W^0(t)$ & $0$ \\
		\hline
		\end{tabular}
	\caption{Hodge Numbers for $G_{2, (m_1,m_2)}^1=x_1^{m_1}+x_2^{m_2}+x_1^{-m_1}$}
	\label{Tn=2j=1}
\end{table}	

\begin{table}[h]
	\centering
		\begin{tabular}{||c||c|c|c|c||}
\hline
		$k$ & $0, 1, \ldots, m_1{m_2}-1 $ & $m_1{m_2}$ & $m_1{m_2}+t; 0 < t <m_1{m_2}$ & $2m_1{m_2}$ \\ \hline
		$H^2(k)$ & $4W^0(k)-4$ & $4$ & $8-4W^0(t)$ & $0$ \\
		\hline
		\end{tabular}
	\caption{Hodge Numbers for $G_{2, (m_1,m_2)}^2=\sum_{i=1}^2(x_i^{m_i}+x_i^{-m_i})$}
	\label{Tn=2j=2}
\end{table}

\subsection{Kloosterman variant Laurent polynomials} \label{Sklo}

Fix $n \in \NN$, $\vec{m} = (m_1, \ldots, m_n) \in \NN^n$ and $1 \leq j \leq n$.  
In this section, let $\Delta$ denote the polytope of the Laurent polynomial 
\[K_{n,\vec{m}}^j=x_1^{m_1} + \cdots + x_n^{m_n} + (x_1 \cdots x_j)^{-1}.\]

\begin{figure}[h]
	\centering
		\includegraphics{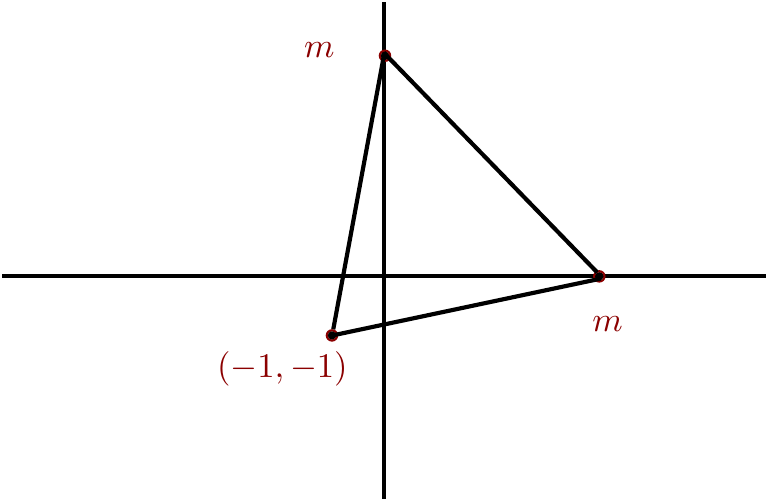}
	\caption{The polytope $\Delta(K^2_{2,(m,m)})$}
	\label{fig:example3}
\end{figure}

The vertices of $\Delta$ are $-\vec{1}_j:=-\sum_{i=1}^j e_i = (-1, \ldots,-1, 0, \ldots,0)$ and $v_1, \ldots, v_n$.
The cone is $c(\Delta) = \{(x_1, \ldots, x_n) \in \R^n \mid x_i \geq 0, \ j+1 \leq i \leq n\}$.

The vectors with initial point $-\vec{1}_j$ along the edges of $\Delta$ are, for $1\leq \ell \leq n$,
\[w_\ell = \sum_{i=1}^j e_j + m_\ell e_\ell.\]
The volume of the polytope $\Delta=\Delta(K^j_{n,\vec{m}})$ is $V(\Delta)=\frac{1}{n!} \det(w_1, \ldots, w_n)$.
Write $s_k$ for the $k$th symmetric product in $m_1, \ldots, m_j$.  Then 
\[V(\Delta(K_{n, \vec{m}}^j))=[s_j + \sum_{i=1}^{j-1} (-1)^i i s_{j-1-i}]\prod_{i={j+1}}^n m_i/n!.\]
The denominator of $\Delta$ is $D={\rm LCM}(m_1, \ldots, m_n)$.

\begin{lem} \label{LdetK}
\begin{enumerate}
\item Suppose $1 \leq \ell \leq j$.  Let $\delta_\ell$ be the face of $\Delta$ containing the vertices $-\vec{1}_j$ and $v_i$ for $1 \leq i \leq n$ and $i \not = \ell$.
Then $\delta_\ell$ is contained in the hyperplane: 
\[\sum_{i \not = \ell} \frac{1}{m_i} x_i - (1 + \sum_{i \not = \ell} \frac{1}{m_i})x_{\ell}=1.\]
\item The other faces of $\Delta$ are contained in the hyperplanes $\sum_{i=1}^n \frac{1}{m_i} x_i =1$ and $x_i=0$ for $j+1 \leq i \leq n$.
\end{enumerate}
\end{lem}

\subsubsection{General dimension, equilateral case}

Suppose $\vec{m}=(m_1, \ldots, m_n)$ and write $K_{n, m}^j:= K_{n, \vec{m}}^j$.

\begin{prop}
For $0 \leq k \leq nm$, the weight numbers for $K_{n,m}^j$ are:
\[W(k)=\dbinom{n-1+k}{n-1} + \sum_{s=1}^{j}\beta(j,s)\sum_{\ell=1}^n\dbinom{k-\ell m+(n-j-1)}{n-s-1}+\alpha(j,k),\]
where $\beta(j,s) =\dbinom{j}{s}$ unless $j=s=n$ in which case $\beta(n,n)=0$
and $\alpha(j,k) =0$ unless $j=n$ and $0 < k \equiv 0 \bmod m$ in which case $\alpha(j,k)=1$.
\end{prop}

\begin{proof}
The lattice points $\{(x_1, \dots, x_n)\in \Z^n \mid x_i \geq 0\}$ have the same weight as in the diagonal case.
This contributes $\dbinom{n-1+k}{n-1}$ to $W(k)$.

Thus it suffices to consider the weight of $\vec{x}=(x_1, \dots, x_n)$ when at least one coordinate is negative.   
By symmetry, it suffices to first focus on the points $\vec{x}$ closest to the face $\delta$
of $\Delta$ containing the vertices $-\vec{1}_j$ and $v_i$ for $2 \leq i \leq n$.
This face is contained in the hyperplane
\[\frac{1}{m} \sum_{i=2}^n x_i -  \frac{m+(j-1)}{m}x_{1} = 1.\]
These points satisfy the conditions:
$x_1 <0$ and $x_i \geq x_1$ for $2 \leq i \leq j$, and $x_i \geq 0$ for $j+1 \leq i \leq n$.  

The condition $k \leq nm$ implies that $x_1 \in \{-1, \dots, -n\}$.  Fix $-\ell\in \{-1, \dots, -n\}$ and let $x_1 = -\ell$.
First suppose $x_i > x_1$ for all $2 \leq i \leq j$.  The smallest weight $k$ possible for this set of points is
\[
-(j-1)(\ell-1) + (m+j-1)(\ell) +(n-j)(0)= m\ell+j-1,
\]
occurring when $x_i = -(\ell-1)$ for $2 \leq i \leq j$ and $x_i=0$ for $j+1 \leq i \leq n$.  
To increase this value to $k$, one needs to add a combined total of $k - (m\ell+j-1)$ to $\{x_i \mid i \geq 2\}$.  
For $1\leq \ell \leq n$, there are 
\[
\dbinom{k - (m\ell+j-1)+(n-2)}{n-2}=\dbinom{k - m\ell+n-j-1}{n-2}
\]
ways to do this, which is the number of
points of weight $k$ with $x_1 = -\ell$,  $x_i > x_1$ for $2 \leq i \leq j$ and $x_i \geq 0$ for $j+1 \leq i \leq n$.

Next, let $2 \leq s \leq j$ and suppose $\#\{i \leq j \mid x_i = -\ell\}=s$.  Recall that $-\ell \in \{-1, \dots, -n\}$.
(This is the case where $\vec{x}$ is equidistant to more than one face of $\Delta$ containing $-\vec{1}_j$.)  
For ease of notation, suppose $x_i = - \ell$ for $1 \leq i \leq s$.  
Recall that $-\ell \in \{-1, \dots, -n\}$, and $x_i \geq -(\ell-1)$ for $s+1 \leq i \leq j$ and $x_i >0$ for $j+1 \leq i \leq n$.  
The smallest weight $k$ possible for this set of points is
\[
-(j-s)(\ell-1) -(s-1)(\ell) + (m+j-1)(\ell) +0(n-j)= m\ell+j-s,
\]
occurring when $x_i = -(\ell-1)$ for $s+1 \leq i \leq j$ and $x_i=0$ for $j+1 \leq i \leq n$.  
To increase this value to $k$, one needs to add a combined total of $k - (m\ell+j-s)$ to $\{x_i \mid i \geq s+1\}$.  
Thus, for $1\leq \ell \leq n$, outside the case $s=j=n$, there are 
\[
\dbinom{k - (m\ell+j-s)+(n-s-1)}{n-s-1}=\dbinom{k - m\ell+n-j-1}{n-1-s}
\]
ways to do this, which is the number of
points of weight $k$ with $x_i = - \ell$ for $1 \leq i \leq s$, and $x_i> -\ell$ for $s+1 \leq i \leq j$, and $x_i \geq 0$ for $j+1 \leq i \leq n$.
Let $C_s(k)$ denote the set of lattice points $\vec{x}$ of weight $k$ such that 
$\#\{i \mid x_i ={\rm min}(x_1, \ldots, x_n)\}=s$.
The conclusion is that, outside the case $s=j=n$,
\[
\#C_s(k)=\dbinom{j}{s}\sum_{\ell=1}^{n}\dbinom{k - m\ell+n-j-1}{n-s-1}.
\]

If $s=j=n$, none of the sets $C_s$ include the points $\vec{x}$ which are a multiple of $-\vec{1}_j$.  
There is one such point of weight $m\ell$ for each $1\leq \ell \leq n$.  
This contributes one point of weight $k$ only when $0 < k \equiv 0 \bmod m$.
This is accounted for by the definitions of $\beta(j,s)$ and $\alpha(j,k)$.
\end{proof}


\begin{exa}\label{ex.K2} 
Let $n=2$. The difference between the number of lattice points of weight $k/m$ for $K^2_{2,m}$ and $G^0_{2,m}$ is zero if $0 \leq k < m$,
is one if $k=m$, is two if $m < k < 2m$, and is three if $k=2m$.

\begin{table}[h]
	\centering
		\begin{tabular}{||c||c|c|c|c|c|c|c|c|c|c||}
\hline
		$k$ & $0$ & $1$ & $\ldots$ & $m-1$ & $m$ & $m+1$ & $m+2$ &$\ldots$ & $2m-1$ & $2m$ \\ \hline
		$W(k)$ & $1$ & $2$ & $\ldots$ & $m$  & $m+2$ & $m+4$ &  $m+5$ & $\ldots$ & $2m+2$ & $2m+4$ \\ \hline
		$H(k)$ & $1$ & $2$ & $\ldots$ & $m$  & $m$ &  $m$ & $m-1$ & $\ldots$ & $2$ & $1$\\
		\hline
		\end{tabular}
	\caption{Hodge Numbers for $K^2_{2,m}=x_1^{m}+x_2^m+(x_1x_2)^{-1}$}
	\label{HodgeK2}
\end{table}
\end{exa}

\begin{exa}\label{ex.K3}
Let $n=3$.  Table \ref{HodgeK3} shows the difference $\tau(k,m)$ between the number of lattice points of weight $k/m$ for $K^3_{3,m}$ and $G^0_{3,m}$.
\begin{table}[h]
	\centering
		\begin{tabular}{||c||c|c|c|c|c|c||}
\hline
$k$ & $0 \leq k < m$ & $k=m$ & $k=m+\epsilon$ & $k=2m$ & $k=2m + \epsilon$ & $k=3m$\\
& & & $0 < \epsilon < m$ & & $0 < \epsilon < m$ &\\
\hline
$\tau(k,m)$ & $0$ & $1$ & $3\epsilon$ & $3m+1$ & $3m + 6\epsilon$ & $9m+1$\\
\hline
	\end{tabular}
	\caption{The difference between $W_{K^3_{3,m}}(k)$ and $W_{G^0_{3,m}}(k)$}
	\label{HodgeK3}
\end{table}
\end{exa}

\subsubsection{Dimension two, non-equilateral case} Let $n=2$ and ${\rm gcd}(m_1,m_2)=1$.
Recall that one can explicitly compute the weight number for the diagonal polytope $x_1^{m_1} + x_2^{m_2}$ from Theorem \ref{explicitformula} and Table \ref{Tadiag}.
In this section, we denote that weight number by $d(k)$.

Let $\Delta^2$ denote the polytope of $K^2_{2,(m_1,m_2)} = x_1^{m_1}+x_2^{m_2}+(x_1x_2)^{-1}$.
It has vertices at $(m_1,0),(0,m_2),$ and $(-1,-1)$ and denominator $D =m_1m_2$.
Similarly, let $\Delta^1$ denote the polytope of $K^1_{2,(m_1,m_2)} = x_1^{m_1}+x_2^{m_2}+(x_1)^{-1}$
which has vertices at $(m_1,0), (0,m_2)$ and $(-1,0)$ and denominator $D =m_1m_2$.

We compute the weight numbers $W(k)$ for $K^j_{2,(m_1,m_2)}$ for $j = 1,2$.

\begin{prop}
Let $d(k)$ denote the weight number for the diagonal polytope $x_1^{m_1} + x_2^{m_2}$. 
For $0 \leq k \leq 2m_1m_2$ and $1 \leq j \leq 2$, the weight numbers for $K^j_{2,(m_1,m_2)}$ are
\[W_j(k)=d(k) + 
\begin{cases}
1+j& \text{ if } k=2m_1m_2,\\
1& \text{ if } m_1m_2 \leq k < 2m_1m_2 \text{ and } {\rm gcd}(k, m_1m_2)>1,\\
0& \text{ otherwise.} \\
\end{cases}
\]
\end{prop}

\begin{proof}
The contribution to $W_j(k) - d(k)$ comes from points $(x_1,x_2)$ with at least one negative coordinate. 
Then $x_1, x_2 \geq -2$ since $k \leq 2m_1m_2$. 
 
When $j=2$, there are $2 + m_1 +m_2$ new lattice points having at least one negative coordinate: $(-1,-1)$ with weight 1, $(-2,-2)$ with weight 2, 
$(-1, \ell)$ for $0 \leq \ell < m_2$ and $(\ell, -1)$ for $0 \leq \ell < m_1$.
The equation of the face $\delta$ through the vertices $(-1,-1)$ and $(m_1, 0)$ is $\frac{1}{m_1} x_1 - \frac{m_1+1}{m_1} x_2 =1$.
Using this, the weight of $(\ell, -1)$ is $k/m_1m_2$ with $k=m_1m_2+(\ell+1)m_2$.
Similarly, the weight of $(-1, \ell)$ is $k/m_1m_2$ with $k = m_1m_2 + (\ell+1)m_1$.

When $j=1$, there are exactly $m_{2}+2$ new lattice points having at least one negative coordinate: $(-1,0)$ with weight 1, $(-2,0)$ with weight 2, 
and the points $(-1,\ell)$ for $1 \leq \ell \leq m_2$.  Using the equation $-x_1+\frac{1}{m_2} x_2 = 1$ of the corresponding face $\delta$, 
the weight of $(-1,\ell)$ is $k/m_1m_2$ with $k=m_1m_2+m_1\ell$.
\end{proof}

\bibliographystyle{plain}	
\bibliography{WINZetaHodgefinal}		

\end{document}